\documentclass{amsart}
\usepackage{enumerate}
\usepackage{tikz}
\usetikzlibrary{automata,positioning}
\usepackage{hyperref}
\usepackage{mathrsfs}
\usepackage{amsfonts}
\usepackage{latexsym}
\usepackage{epsfig}
\usepackage{indentfirst}
\graphicspath{{figures/}}
\usepackage{amsmath}
\usepackage{amsthm}
\usepackage{graphicx}
\usepackage{amssymb}
\usepackage{amscd}
\usepackage{latexsym}
\usepackage{color}
\DeclareGraphicsExtensions{.eps}
\input xy
\xyoption{all}

\newtheorem{thm}{Theorem}[section]
\newtheorem{lem}[thm]{Lemma}

\newtheorem{prop}[thm]{Proposition}

\theoremstyle{definition}
\newtheorem{defn}[thm]{Definition}

\newtheorem{que}[thm]{Question}

\theoremstyle{remark}

\numberwithin{equation}{section}

\numberwithin{equation}{section} \makeatletter

\begin{document}

\title[Iterated Mazur pattern satellite knots]{An infinite-rank summand from iterated Mazur pattern satellite knots}
\author{Wenzhao Chen}
\address{Max Planck Institute of Mathematics, Vivatsgasse 7, 53111 Bonn, Germany}
\email{chenwenz@msu.edu}
\thanks{}


\keywords{}

\begin{abstract}
We show there exists a topologically slice knot $K$ such that the knots $\{M^n(K)\}_{n=0}^\infty$ obtained by iterated satellite operations by the Mazur pattern span an infinite-rank summand of the smooth knot concordance group. This answers a question raised by Feller-Park-Ray. 
\end{abstract}
\maketitle
\section{Introduction}
A knot $K$ in the 3-sphere is called \emph{smoothly slice} if it bounds a smoothly embedded disk in the 4-ball. Two oriented knots $K_0$ and $K_1$ are \emph{smoothly concordant} if $K_0\# -K_1$ is smoothly slice, where $-K_1$ is the mirror image of $K_1$ with reversed orientation. With the binary operation induced by the connected sum, the set of concordance classes of knots form an abelian group $\mathcal{C}$ called the \emph{smooth knot concordance group}. One can similarly define \emph{topologically slice knots} by using locally flat and topologically embedded disks, and define \emph{exotically slice knots} by using 4-balls endowed with possibly exotic smooth structures. Correspondingly, we have the \emph{topological knot concordance group} $\mathcal{C}_{top}$ and the \emph{exotic knot concordance group} $\mathcal{C}_{ex}$.


The \emph{satellite operation} is a method of constructing knots and induces self-maps of the aforementioned knot concordance groups. Let $P$ be an oriented knot in the solid torus $S^1\times D^2$. For any knot $K$, the knot $P(K)$ is the image of $P$ under a homeomorphism from $S^1\times D^2$ to a regular neighborhood of $K$ that identifies $S^1\times \{*\in\partial D^2\}$ with a Seifert longitude of $K$. $P(K)$ is called a satellite knot with pattern $P$ the and companion $K$. 

Throughout the paper, we will let $M$ denote the Mazur pattern; see Figure \ref{Figure, the Mazur pattern}. The primary focus of this paper is to answer a question raised by Feller-Park-Ray regarding iterated Mazur pattern satellite knots. 
\begin{que}[Question 1.7 of \cite{FPR19}]\label{question by Feller-Park-Ray}
Does there exist a knot $K$ such that $\{M^n(K)\}_{n=0}^{\infty}$ are linearly independent in the smooth knot concordance group $\mathcal{C}$, or better, form a basis for an infinite-rank summand of $\mathcal{C}$? If yes, can $K$ be chosen to be topologically slice?
\end{que}  

It is known that $M^n$ induce distinct operators on $\mathcal{C}_{ex}$ and hence on $\mathcal{C}$. Ray proved this by showing there exists a knot $K$ such that $\{M^n(K)\}_{n=0}^{\infty}$ have distinct Ozsv\'ath-Szab\'o $\tau$-invariants \cite{MR3391056}. Moreover, the Mazur pattern is a strong winding-number-one pattern, which induces an injective map on $\mathcal{C}_{ex}$ \cite{MR3286894}. Ray's result hence suggests the operators $M^n$ provide a fractal structure on $\mathcal{C}_{ex}$ (i.e.\ $M^n$ are ever-shrinking injective maps) \cite{MR2836668}. Later Levine improved this fractal analogy by showing the images of $M^n$ are strictly decreasing as $n$ increases \cite{Lev16}. Question \ref{question by Feller-Park-Ray} further asks whether $M^n$ give rise to linearly independent operators on $\mathcal{C}$. Our main theorem is an affirmative answer to Question \ref{question by Feller-Park-Ray}. 
\begin{figure}[htb!]
\begin{center}
\includegraphics[scale=0.40]{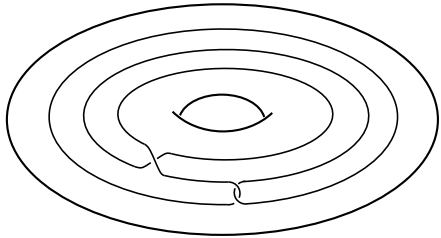}
\caption{The Mazur pattern $M$.}\label{Figure, the Mazur pattern}
\end{center}
\end{figure}
\begin{thm}\label{Main theorem}
Let $D$ be the Whitehead double of the right-hand trefoil, and let $D_2=D\#D$. Then the knots $\{M^n(D_2)\}_{n=1}^{\infty}$ form a basis of an infinite-rank summand of the smooth knot concordance group. 
\end{thm}

The invariants we use to prove Theorem \ref{Main theorem} are the $\varphi_i$-invariants due to Dai-Hom-Stoffregen-Truong, which are defined using a version of knot Floer homology $CFK_\mathcal{R}$ over the ring $\mathcal{R}=\mathbb{F}[U,V]/(UV)$. This version of knot Floer homology is equivalent to the bordered Heegaard Floer invariant $\widehat{CFD}$ of the knot complement \cite{LOT18, hanselman2019cabling}. Hence in theory one can compute these invariants for $M^n(D_2)$ using bimodules in bordered Heegaard Floer homology \cite{lipshitz2015bimodules}. However, the bimodule of $S^1\times D^2-M$ is quite big and hence makes this approach very involved; see Theorem 3.4 of \cite{Lev16} for the bimodule of $S^1\times D^2-M$. Instead, we partially compute the hat-version knot Floer chain complex of $M^n(D_2)$ using the immersed-curve techniques developed in \cite{chen2019knot}, and use it to constrain possible local equivalence classes for $CFK_\mathcal{R}(M^n(D_2))$, and finally, we determine the local equivalence classes using obstructions for knot Floer chain complexes to be realized by knots. This in turn determines the $\varphi_i$-invariants.

%

Intuitively speaking, general satellite operations are quite different from the connected sum and hence often disregard the group structure of $\mathcal{C}$. This feature has been exploited widely to construct linearly independent knots in various context (e.g.\ \cite{CFHH13, MR2496049,DHST19,HKL16,MR3784230}). It is hence reasonable to expect Theorem \ref{Main theorem} to be true, and it is even tempted to believe general iterated satellite knots are independent. As partial evidence to this expectation, the author proved in \cite{ChenThesis19} that if we fix a slice pattern with winding number greater than or equal to two, or certain winding-number-one patterns including the Mazur pattern, then there exists a topologically slice companion knot such that infinitely many knots of the iterated satellites are linearly independent. Despite having such an intuition, many interesting questions on iterated satellite knots remain open and serve as strong testing grounds for our ability to understand knot concordance. For example, are there linearly independent iterated Mazur pattern satellite knots in the topological category? Are the iterated Whitehead doubles independent \cite{MR3749499}? For any winding-number-zero pattern $P$ which induces a non-constant map on $\mathcal{C}$, are the associated graded groups of the \emph{$P$-filtration} $\mathcal{C}\supseteq\langle \text{Im}P \rangle\supseteq \langle \text{Im}P^2 \rangle \supseteq \cdots$ always of infinite rank \cite{hedden2018satellites}? Perhaps more classically, if we allow the companion and pattern knots to vary, are algebraic knots linearly independent \cite{Rud76}?  
\paragraph*{\textbf{Organization.}}Section \ref{Section, Preliminary} reviews the necessary background of Heegaard Floer homology needed for this paper. The proof of the main theorem is given in Section \ref{Subsection, main argument modulo a technical lemma}, modulo a computational result which is proved in Section \ref{Subsection, computational result}.

\paragraph*{\textbf{Acknowledgment}} This work would not be possible without the encouragement of Matt Hedden. I also thank Arunima Ray for helpful conversations. The author is grateful to the Max Planck Institute for Mathematics in Bonn for its hospitality and financial support.

\section{Preliminary}\label{Section, Preliminary} 
\subsection{The knot Floer chain complexes}
The knot Floer homology refers to a package of knot invariants introduced by Ozsv\'ath-Szab\'o \cite{OS04a} and independently by Rasmussen \cite{Ras03}. In this subsection we briefly recall the algebraic aspect of the knot Floer homology package, following the convention in \cite{DHST19} and \cite{MR3950650}. For further details and other aspects of this theory, we refer the interested readers to the aforementioned papers as well as survey papers \cite{MR3604497, MR3591644, MR3838884} .

For any knot $K$ in $S^3$ together with some choice of auxiliary data $\mathfrak{d}$, the knot Floer chain complex $\mathcal{CFL}^{-}(K,\mathfrak{d})$ is a $\mathbb{Z}\oplus\mathbb{Z}$-graded, freely and finitely generated chain complex over $\mathbb{F}[U,V]$, where the bigrading is compatible with the 
$\mathbb{F}[U,V]$-module structure: Let $\mathfrak{S}$ denote a set of generators over $\mathbb{F}[U,V]$ which is homogeneous with respect to the bigrading $gr=(gr_U, gr_V)$. Then the bigrading on $\mathcal{CFL}^{-}(K)$ is determined by $gr:\mathfrak{S}\rightarrow\mathbb{Z}\oplus\mathbb{Z}$ together with the rules that multiplication by $U$ is $(-2,0)$ graded, and multiplication by $V$ is $(0,-2)$ graded.

Different choices of the auxiliary data give rise to homotopy equivalent bigraded chain complexes over $\mathbb{F}[U,V]$. Therefore, we simply denote the knot Floer chain complex by $\mathcal{CFL}^{-}(K)$, understood up to chain homotopy equivalence.

The grading $gr_U$ is also called \emph{the Maslov grading}, and we define \emph{the Alexander grading} $A(\cdot):=\frac{1}{2}(gr_U(\cdot)-gr_V(\cdot))$.

We need two variations of $\mathcal{CFL}^{-}(K)$. To define the first variation, let $\mathcal{R}=\mathbb{F}[U,V]/(UV)$. Define $CFK_\mathcal{R}(K)$ to be $\mathcal{CFL}^{-}(K)\otimes_{\mathbb{F}[U,V]}\mathcal{R}$, which inherits the bigrading from $\mathcal{CFL}^{-}(K)$. The second variation is the hat-version knot Floer chain complex $\widehat{CFK}(K)$. It is a $\mathbb{Z}$-filtered, $\mathbb{Z}$-graded chain complex over $\mathbb{F}$. Up to filtered chain homotopy equivalence, it can be obtained from $CFK_\mathcal{R}(K)$ via the following procedure: Let $\{x_i\}$ be a homogeneous, filtered basis for $CFK_\mathcal{R}(K)$ over $\mathcal{R}$. The underlying vector space for $\widehat{CFK}(K)$ space is generated by $\{x_i\}$ over $\mathbb{F}$ and the differential is obtained from the differential of $CFK_{\mathcal{R}}(K)$ by setting $U=0$ and $V=1$. The elements in $\{x_i\}$ inherit the Maslov grading and the Alexander grading from $CFK_\mathcal{R}(K)$. The $\mathbb{Z}$-grading on $\widehat{CFK}(K)$ is the Maslov grading. The $\mathbb{Z}$-filtration $\cdots \subseteq\mathcal{F}_d\subseteq \mathcal{F}_{d+1}\subseteq \cdots$ on $\widehat{CFK}(K)$ is given by $\mathcal{F}_d=Span_\mathbb{F}\{x_i \vert A(x_i)\leq d\}$.
\subsection{The Dai-Hom-Stoffregen-Truong concordance homomorphisms $\varphi_j$}
In this subsection, we briefly introduce the knot concordance homomorphisms $\varphi_j$ due to Dai-Hom-Stoffregen-Truong, defined using $CFK_{\mathcal{R}}(K)$. We begin with recalling that $CFK_\mathcal{R}(K)$ belongs a class of $\mathcal{R}$-complexes called \emph{knot-like complexes}.
\begin{defn}[Definition 3.1 of \cite{DHST19}]
A bigraded, free, finitely generated chain complex $(C,gr_U,gr_V)$ over $\mathcal{R}$ is called a knot-like complex if
\begin{itemize}
\item[(1)]$H_*(C/U)/V-torsion$ is isomophic to $\mathbb{F}[V]$ and is supported in $gr_U=0$.
\item[(2)]$H_*(C/V)/U-torsion$ is isomophic to $\mathbb{F}[U]$ and is supported in $gr_V=0$.
\end{itemize}
\end{defn}
\emph{Standard complexes} are knot-like complexes of particular interest to us. To define them, first recall an $\mathcal{R}$-complex $(C_{\mathcal{R}}, \partial)$ is called \emph{reduced} if $\partial \equiv 0 \mod \ (U,V)$, and in this case we decompose its differential $\partial=\partial_U+\partial_V$, where $\partial_U(x) \in \text{Im} U$ and $\partial_V(x)\in \text{Im} V$ for any $x\in C_{\mathcal{R}}$.
\begin{defn}[Definition 4.3 of \cite{DHST19}]
Let $n\in 2\mathbb{N}$. Given $n$ nonzero integers $(b_1,\ldots,b_n)$, the standard complex $C(b_1,\ldots,b_n)$ is the reduced complex freely generated over $\mathcal{R}$ by $\{x_0,x_1,\ldots,x_n\}$ as an $\mathcal{R}$-module and with differentials as follows. For $i$ odd, 
$$
\begin{aligned}
\partial_U x_{i-1}= U^{|b_i|}x_i \ if\ b_i<0\\
\partial_U x_{i}= U^{b_i}x_{i-1} \ if\ b_i>0,
 \end{aligned}
$$
and for $i$ even, 
$$
\begin{aligned}
\partial_V x_{i-1}= V^{|b_i|}x_i \ if\ b_i<0\\
\partial_V x_{i}= V^{b_i}x_{i-1} \ if\ b_i>0.
 \end{aligned}
$$
\end{defn}
 
In \cite{DHST19}, an equivalence relation called \emph{local equivalence} is defined for knot-like complexes. We will not need the precise definition of local equivalence. Instead, we recall the following important result. 
\begin{thm}[Theorem 6.1 and Corollary 6.2 of \cite{DHST19}]\label{Theorem, every knot-like complex has a standard representative}
Every knot-like complex $C$ is locally equivalent to a standard complex $C(a_1,\ldots, a_n)$, and in this case $C$ is homotopic equivalent to $C(a_1,\ldots, a_n)\oplus A$ for some $\mathcal{R}$-complex $A$.
\end{thm}  
Therefore, the local equivalence class of a knot-like complex is determined by a finite sequence of integers. For a knot $K$, we define the $a_i$-invariants for $K$ as the sequence of integers corresponding to the local equivalence class of $CFK_{\mathcal{R}}(K)$. The set of $a_i$'s for a knot $K$ are symmetric as described by the following proposition. 

\begin{prop}[Lemma 6.10 of \cite{DHST19}]\label{Proposition, a_i's for a knot are symmetric}
Let $K$ be a knot in $S^3$, and let $C = C(a_1, \ldots , a_n)$ be the standard complex locally equivalent to $CFK_\mathcal{R}(K)$. Then $C$ is symmetric, i.e.\ $a_i=-a_{n-i+1}$.
\end{prop}
The local equivalence relation is a partial algebraic model for concordance. If two knots $K_1$ and $K_2$ are smoothly concordanct, then $CFK_{\mathcal{R}}(K_1)$ is locally equivalent to $CFK_{\mathcal{R}}(K_1)$. Therefore, the $a_i$'s for knots are concordance invariants. More surprisingly, the $a_i$'s can be used to construct concordance homomorphisms.
\begin{defn}[Definition 7.1 of \cite{DHST19}]
Given a knot-like complex $C_\mathcal{R}$. Assume $C_{\mathcal{R}}$ is locally equivalent to the standard complex $C(a_1,\ldots,a_n)$. For any positive integer $j$, define $$\varphi_j(C_\mathcal{R})=\#\{a_j\vert a_i=j,\ i\ \text{odd}\}-\#\{a_j\vert a_i=-j,\ i\ \text{odd}\}.$$ For a knot $K$ in $S^3$, define $\varphi_j(K)=\varphi_j(CFK_{\mathcal{R}}(K))$.  
\end{defn} 

\begin{thm}[cf.\ Theorem 1.1 of \cite{DHST19}]
For any positive integer $j$, the map $\varphi_j$ induces a group homomorphism from the knot concordance group $\mathcal{C}$ to $\mathbb{Z}$. 
\end{thm}

\section{Proof of Theorem \ref{Main theorem}}\label{Section, Proof}

We prove Theorem \ref{Main theorem} by proving the following stronger theorem.
\begin{thm}\label{Theorem, epsilon equivalence class of M^n(K)}
$CFK_{\mathcal{R}}(M^n(D_2))$ is locally equivalent to $C(1,-n-1,n+1,-1)$ for $n=0,1,2,\ldots$.
\end{thm}
We first show how Theorem \ref{Theorem, epsilon equivalence class of M^n(K)} implies Theorem \ref{Main theorem}.
\begin{proof}[Proof of Theorem \ref{Main theorem}]
Let $\varphi_i:\mathcal{C}\rightarrow \mathbb{Z}$, $i=1,2,\ldots$ be the concordance homomorphisms due to Dai-Hom-Stoffregen-Truong. When $n\geq 1$, by Theorem \ref{Theorem, epsilon equivalence class of M^n(K)}, we have $\phi_i(M^n(D_2))=1$ if $i=1$ or $i=n+1$, and $\varphi_i(M^n(D))=0$ otherwise. Therefore, $\oplus_{i=2}^{\infty}\varphi_i$ maps the span of $\{M^n(D_2)\}_{n=1}^{\infty}$ isomorphically to $\oplus_{i=2}^{\infty} \mathbb{Z}$. Hence the span of $\{M^n(D_2)\}_{n=1}^{\infty}$ form a summand of the smooth knot concordance group $\mathcal{C}$.
\end{proof}

The proof of Theorem \ref{Theorem, epsilon equivalence class of M^n(K)} occupies the rest of the section. For ease of exposition, we will divide the proof into two subsections: the main argument of the proof is given in Subsection \ref{Subsection, main argument modulo a technical lemma}, modulo the proof of a computational result which is given in Subsection \ref{Subsection, computational result}.  
\subsection{Proof of Theorem \ref{Theorem, epsilon equivalence class of M^n(K)} modulo a lemma}\label{Subsection, main argument modulo a technical lemma}
The overall argument will be inductive on $n$. We will not directly compute $CFK_{\mathcal{R}}(M^n(D))$ as this would be too involved. Instead, we will use a constraint derived from $\widehat{CFK}(M^n(D))$ to determine the $a_i$'s. The advantage is that $\widehat{CFK}(M^n(D))$ is easier to compute than $CFK_{\mathcal{R}}(M^n(D))$. We first describe this constraint below. 

Let $(\widehat{C}, M, \{\mathcal{F}_d\})$ be a finitely generated, $\mathbb{Z}$-graded, $\mathbb{Z}$-filtered chain complex over $\mathbb{F}$, where $M$ denotes the $\mathbb{Z}$-grading and $\{\mathcal{F}_d\}$ denotes the filtration:
$$\cdots\subseteq\mathcal{F}_{d-1}\subseteq\mathcal{F}_{d}\subseteq\mathcal{F}_{d+1}\subseteq\cdots$$
Let $A$ denote the filtration grading and let $$\iota(a,m,l)_*:H_{m}(\mathcal{F}_{a}/\mathcal{F}_{a-1})\rightarrow H_{m}(\mathcal{F}_{a+l}/ \mathcal{F}_{a-1})$$ be the map induced by inclusion. 
\begin{defn}\label{Definition, characteristic set}
Let $(\widehat{C}, M, \{\mathcal{F}_d\})$ be a finitely generated, $\mathbb{Z}$-graded, $\mathbb{Z}$-filtered chain complex over $\mathbb{F}$. Define the characteristic multi-set $Ch(\widehat{C})$ of $\widehat{C}$ to be the set consists of triples $(a,m,l)\in \mathbb{Z}\times \mathbb{Z}\times \mathbb{Z}^{+}$ such that $(a,m,l)$ appears $k$ times in $Ch(\widehat{C})$ if and only if $k=dim(\ker \iota(a,m,l)_*)-\dim(\ker \iota(a,m,l-1)_*)$.
\end{defn}

We will be dealing with chain complexes $\widehat{C}$ such that $H_*(\widehat{C})\cong \mathbb{F}$ and $\widehat{C}$ is reduced, i.e.\ the differentials of $\widehat{C}$ strictly decrease the filtration level. In this case, $Ch(\widehat{C})$ can be easily computed in terms of a \textit{vertically simplified basis}. Recall that a basis $\{x_i\}$ for a $\mathbb{Z}$-graded, $\mathbb{Z}$-filtered chain complex is called a \textit{vertically simplified basis} if:
\begin{itemize}
\item[(1)]$\{x_i\}_{i=0}^{2n}$ is a filtered basis.
\item[(2)]Each basis element $x_i$ is homogeneous with respect to the $\mathbb{Z}$-grading.
\item[(3)]For each $x_i$, we either have $\partial x_i=0$ or $\partial x_i=x_{i+1}$.
\end{itemize}
For any $\mathbb{Z}$-graded, $\mathbb{Z}$-filtered chain complex, there exists a vertically simplified basis; see Proposition 11.57 of \cite{LOT18}. As $H_*(\widehat{C}) \cong \mathbb{F}$, we can find a vertically simplified basis $\{x_i\}_{i=0}^{2n}$ for $\widehat{C}$ such that $\partial x_{2i}=0$ for $i=1,\ldots,n$ and $\partial x_{2i-1}=x_{2i}$ for $i=1,\ldots, n$, where $\dim\widehat{C}=2n+1$. Let $l_{2i-1}$ be the length of the arrow $\partial x_{2i-1}=x_{2i}$ for $i=1,\ldots, n$. Then $Ch(\{\widehat{C}\})=\{(A(x_{2i}),M(x_{2i}),l_{2i-1})\}_{i=1}^{n}$, i.e.\ it records the Alexander grading of the target, Maslov grading of the target, and the length of each arrow. 
 
The constraint of $\widehat{CFK}$ on the local equivalence classes of $CFK_\mathcal{R}$ is given in the following lemma.

\begin{lem}\label{Lemma, choices of a_i are limited by characteristic set}
Let $C_\mathcal{R}$ be an $\mathcal{R}$-complex locally equivalent to $C(a_1,\ldots, a_{2n})$, and denote the corresponding hat-version complexes by $\widehat{C_\mathcal{R}}$ and $\widehat{C}(a_1,\ldots,a_{2n})$. Then $Ch(\widehat{C}(a_1,\ldots,a_{2n})) \subset Ch(\widehat{C_\mathcal{R}})$. 
\end{lem}
\begin{proof}
Since $C_\mathcal{R}$ is locally equivalent to $C(a_1,\ldots, a_{2n})$, by Corollary 6.2 of \cite{DHST19} $C_\mathcal{R}$ is homotopy equivalent to $C(a_1,\ldots, a_{2n})\oplus A$, where $A$ is some $\mathcal{R}$-complex. In particular, we have $\widehat{C_\mathcal{R}}$ is $\mathbb{Z}$-filtered homotopy equivalent to $\widehat{C}(a_1,\ldots, a_{2n})\oplus \widehat{A}$. It follows easily from Definition \ref{Definition, characteristic set} that the characteristic set is invariant under filtered chain homotopy equivalence and hence $Ch(\widehat{C_\mathcal{R}})=Ch(\widehat{C}(a_1,\ldots, a_{2n})\oplus \widehat{A})$. This implies $Ch(\widehat{C}(a_1,\ldots, a_{2n}))\subset Ch(\widehat{C_\mathcal{R}})$.
\end{proof}
To apply the above constraint for our purpose, it turns out we will only need partial information on $Ch(\widehat{CFK}(M^n(D)))$; this is summarized in the following lemma. 
\begin{lem}\label{Lemma, length of vertical arrows over simplied basis}
If $CFK_{\mathcal{R}}(K)$ is locally equivalent to $C(1,-n+1,n-1,1)$ for $n\geq 2$, then $CFK_R(M(K))$ is locally equivalent to a reduced complex $C_\mathcal{R}$ such that over a vertically simplified basis for $\widehat{C_\mathcal{R}}$ the following properties hold:
\begin{itemize}
\item[(1)]The vertical arrows with terminals of Alexander grading $-n-1$ and Maslov index $-2n-2$ are of length $1$. 
\item[(2)]The vertical arrows with initials of Alexander grading $n$ and Maslov grading $-1$ are either of length $1$ or $n$, and there is only one such vertical arrow of length $n$.
\item[(3)]The vertical arrows with initials or terminals of Alexander grading $-n+1$ and Maslov grading $-2n$ are of length $1$.
\item[(4)]There are no vertical arrows with initials of Alexander grading $n-2$ and Maslov grading $-3$, and the vertical arrows with terminals of Alexander grading $n-2$ and Maslov grading $-3$ are of length $1$.
\item[(5)]The vertical arrows with terminals of Alexander grading $0$ and Maslov grading $-2$ are either of length $1$ or $n$, and the vertical arrows with initials of Alexander grading $0$ and Maslov grading $-2$ are of length $1$.
\item[(6)]The vertical arrows with terminals of Alexander grading $1$ and Maslov grading $-1$ are of length $1$.
\item[(7)]There are no vertical arrows with initials of Alexander grading $-2$ and Maslov grading $-4$, and the vertical arrows with terminals of Alexander grading $-2$ and Maslov grading $-4$ are of length $1$.
\item[(8)]The vertical arrows with initials or terminals of Alexander grading $-1$ and Maslov grading $-3$ are of length $1$.
\item[(9)]There are no vertical arrows with terminals of Alexander grading $2$ and Maslov grading $0$ and the vertical arrows with initials of Alexander grading $2$ and Maslov grading $0$ are of length $1$.
\end{itemize}
\end{lem}


We defer the proof of Lemma \ref{Lemma, length of vertical arrows over simplied basis} until Section \ref{Subsection, computational result}, hoping the reader might feel less digressed by a detailed computation.

The last ingredient we need is a lemma that concerns the realizability of certain $\mathcal{R}$-complexes by knots. This is similar to Lemma 3.8 of \cite{Hom15}; we remind the readers that the convention of signs of the $a_i$-invariants in \cite{Hom15} are different than what we use here.  

\begin{lem}\label{Lemma, realizability of certain complexes}
The following standard complexes are not locally equivalent to $CFK_{\mathcal{R}}(K)$ for any knot $K$ in $S^3$.
\begin{itemize}
\item[(1)] $C(1,-1,-1,\ldots)$.
\item[(2)] $C(1,-1,1,1,\ldots)$.
\item[(3)] $C(1,-n,-1,-l,\ldots)$ for any $n>0$ and any $l>0$.
\item[(4)] $C(1,-n,-1,1,1,\ldots)$ for any $n>0$.
\item[(5)] $C(1,-n,1,1,\ldots)$ for any $n>0$.
\item[(6)] $C(1,-n,1,-1,-1,\ldots)$ for any $n>0$.
\end{itemize}
\end{lem}
\begin{proof}
The proofs for item (1), (2), and (5) can be found in Lemma 3.8 of \cite{Hom15}. In fact, the proofs of all of the items are similar, so we will only prove (3) and leave the other items for the readers.  

We prove this by contradiction. Let $K$ be a knot in $S^3$. We abbreviate $(\mathcal{CFL}^{-}(K),\partial)$ by $(C,\partial)$. Let $C_\mathcal{R}$ be the $\mathcal{R}$-complex constructed from $C$. Assume $C_\mathcal{R}$ is locally equivalent to $C'=C(1,-n,-1,-l,\ldots)$ for some $n>0$ and some $l>0$. 

By Corollary 6.2 of \cite{DHST19}, $C_\mathcal{R}$ is isomorphic to $C' \oplus A$. Let $\{x_i\}$ be a homogeneous basis for $C_\mathcal{R}$ over $\mathcal{R}$ extending the standard basis for $C'$. Then $\{x_i\}$ induces a filtered basis for $C$ over $\mathbb{F}[U,V]$, which we still denote by $\{x_i\}$. As $\{x_i\}$ extends the standard basis for $C'$, after reordering the basis if necessary we may assume in $(C,\partial)$ (see Figure \ref{Figure, the complex C(1,-n,-1,-l,...)}):
\begin{itemize}
\item[(1)]$\partial_Ux_1= Ux_0$ and $\partial_Ux_2= Ux_3$.
\item[(2)]$\partial_Vx_1= V^nx_2$. 
\item[(3)]$\partial_Vx_3= V^lx_4$.
\item[(4)]All other differentials to or from $x_0,\ldots,x_3$ have coefficient divisible by $UV$.
\end{itemize}

\begin{figure}[htb!]
\begin{center}
\includegraphics[scale=0.25]{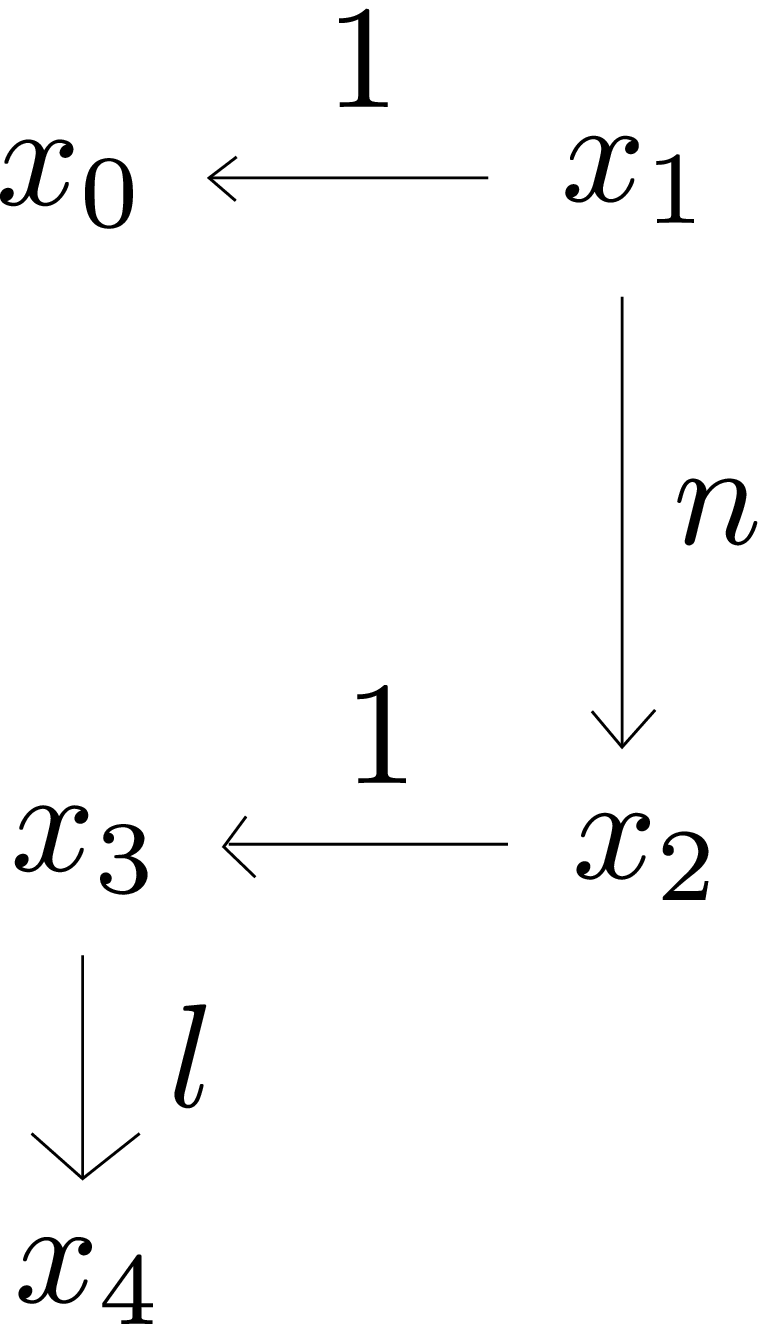}
\caption{}\label{Figure, the complex C(1,-n,-1,-l,...)}
\end{center}
\end{figure}

Note $\partial_U\partial_Vx_1=V^n Ux_3$. In order for $\partial^2 x_1=0$, there must be an arrow from $x_1$ to $Uy$ for some $y=\sum V^{n_i}x_i$ such that $\partial_Vy=V^n x_3$. In particular, there is some other differentials to $x_3$ with coefficient being a power of $V$. This contradicts (4) listed above.

\end{proof}
\begin{proof}[Proof of Theorem \ref{Theorem, epsilon equivalence class of M^n(K)}]
We induct on $n$. For $n=0$, $CFK_{\mathcal{R}}(D_2)$ is locally equivalent to $C(1,-1,1,-1)$ as $CFK_\mathcal{R}(D)$ is locally equivalent to $C(1,-1)$. 

We move to the inductive step. Assume $CFK_{R}(K)$ is locally equivalent to $C(1,-(n-1),(n-1),-1)$ for some $n\geq 2$, and let $C_\mathcal{R}=CFK_{\mathcal{R}}(M(K))$. We prove $C_\mathcal{R}$ is locally equivalent to $C(1,-n,n,-1)$ in six steps. First note that by Theorem 1.4 of \cite{Lev16}, we have $\tau(M(K))=n+1$ and $\epsilon(M(K))=1$. 

\textbf{Step 1.} \emph{$a_1=1$}. Since $\epsilon(M(K))=1$, we have $a_1>0$. Suppose $C_\mathcal{R}$ is locally equivalent ot $C(a_1,\ldots,a_k)$. By the symmetry the $a_i$'s, we have $a_k=-a_1$. Therefore, $a_1\in\{l\vert (-n-1,-2n-2,l)\in Ch(\widehat{C_{\mathcal{R}}})\}$ by Lemma \ref{Lemma, choices of a_i are limited by characteristic set}. By Lemma \ref{Lemma, length of vertical arrows over simplied basis} (1), $a_1=1$. 

\textbf{Step 2.} \emph{$a_2=-n$.} As $a_1=1$, we have $a_2<0$ by Lemma 3.7 of \cite{Hom15}. So by Lemma \ref{Lemma, choices of a_i are limited by characteristic set} and Lemma \ref{Lemma, length of vertical arrows over simplied basis} (2), $a_2=-1$ or $a_2=-n$. 

We show that $a_2\neq -1$ by contradiction. If $a_2=-1$, we first claim $a_3=1$. To see this, note $a_3=-a_{k-3}$ and hence $a_3\in\{l\vert (-(n-1),-2n,l)\in Ch(\widehat{C_{\mathcal{R}}})\}\cup \{-l\vert (-(n-1)-l,-2n-1,l)\in Ch(\widehat{C_{\mathcal{R}}})\}$ by Lemma \ref{Lemma, choices of a_i are limited by characteristic set}. By Lemma \ref{Lemma, length of vertical arrows over simplied basis} (3), we have $a_3=-1$ or $a_3=1$. However, $a_3 \neq -1$ by Lemma \ref{Lemma, realizability of certain complexes} (1). Therefore, we have $a_3=1$. Now similarly we have $a_4\in \{l\vert (n-2,-3,l)\in Ch(\widehat{C_{\mathcal{R}}})\}\cup \{-l\vert ((n-2)-l,-4,l)\in Ch(\widehat{C_{\mathcal{R}}})\}$. By Lemma \ref{Lemma, length of vertical arrows over simplied basis} (4), $a_4=1$. By Lemma \ref{Lemma, realizability of certain complexes} (2), $a_4\neq 1$ and hence we have derived a contradiction. Therefore, $a_2\neq -1$ and hence $a_2=-n$.

\textbf{Step 3.} $a_3\in\{-1,1,n\}$. Note $a_3=-a_{k-3}$ by symmetry, and hence $a_3\in \{l\vert (0,-2,l)\in Ch(\widehat{C_{\mathcal{R}}})\}\cup \{-l\vert (-l,-3,l)\in Ch(\widehat{C_{\mathcal{R}}})\}$. Therefore, $a_3\in\{-1,1,n\}$ by Lemma \ref{Lemma, length of vertical arrows over simplied basis} (5).

\textbf{Step 4.} $a_3\neq -1$. We prove this by contradiction. First we claim that if $a_3=-1$, then $a_4>0$. To see this claim, assume otherwise $a_4=-l<0$. Then $C_\mathcal{R}$ is locally equivalent to $C(1,-n,-1,-l,\ldots)$, which is not realizable by knots by Lemma \ref{Lemma, realizability of certain complexes} (3). Therefore, $a_4>0$.

As $a_4>0$, we have $a_4\in \{l\vert (1,-1,l)\in Ch(\widehat{C}) \}$ by Lemma \ref{Lemma, choices of a_i are limited by characteristic set}. By Lemma \ref{Lemma, length of vertical arrows over simplied basis} (6) we know $a_4=1$. Similarly by Lemma \ref{Lemma, length of vertical arrows over simplied basis} (7) one sees $a_5=1$. However, this contradicts Lemma \ref{Lemma, realizability of certain complexes} (4), which states $C_{\mathcal{R}}$ can not be locally equivalent to $C(1,-n,-1,1,1,\ldots)$. Therefore, $a_3\neq -1$.

\textbf{Step 5.} \emph{$a_3\neq 1$}. We prove this by contradiction. If $a_3=1$, then by Lemma \ref{Lemma, choices of a_i are limited by characteristic set}, $a_4\in \{l\vert (-1,-3,l)\in Ch(\widehat{C_{\mathcal{R}}})\}\cup \{-l\vert (-1-l,-4,l)\in Ch(\widehat{C_{\mathcal{R}}})\}$. By Lemma \ref{Lemma, length of vertical arrows over simplied basis} (8) we have $a_4=1$ or $a_4=-1$. $a_4\neq 1$ by Lemma \ref{Lemma, realizability of certain complexes} (5). Therefore, $a_4=-1$. Then applying Lemma \ref{Lemma, choices of a_i are limited by characteristic set} and Lemma \ref{Lemma, length of vertical arrows over simplied basis} (9) we have $a_5=-1$. However, this contradicts Lemma \ref{Lemma, realizability of certain complexes} (6), which says $C_{\mathcal{R}}$ can not be locally equivalent to $C(1,-n,1,-1,-1,\ldots)$.

\textbf{Step 6.} \emph{$a_i=0$ for $i\geq 5$ and $a_4=-1$.} By Step 3, 4, and 5, we have $a_3=n$. If $a_i\neq 0$ for some $i\geq 5$. Then by symmetry of the $a_i$'s we have $C_{\mathcal{R}}$ is locally equivalent to $C(1,-n,n,\ldots,-n, n,-1)$. By Lemma \ref{Lemma, choices of a_i are limited by characteristic set} this implies $(0,-2,n)$ appears in $Ch(\widehat{C_{\mathcal{R}}})$ at least twice, corresponding to the two vertical arrows of length $n$. However, by Lemma \ref{Lemma, length of vertical arrows over simplied basis} (2) we know $(0,-2,n)$ only appears in $Ch(\widehat{C_{\mathcal{R}}})$ once. Therefore $a_i=0$ for $i\geq 5$ and $a_4=-1$ readily comes from the symmetry of the $a_i$'s. 

\end{proof}

\subsection{Proof of Lemma \ref{Lemma, length of vertical arrows over simplied basis}}\label{Subsection, computational result}
\begin{proof}
Since $CFK_\mathcal{R}(K)$ is locally equivalent to  $C_{n-1}=C(1,-n+1,n-1,1)$, $CFK_\mathcal{R}(K)$ is homotopic equivalent to $C_{n-1}\oplus A_{n-1}$ for some $\mathcal{R}$-complex $A_{n-1}$. Theorem 11.26 of \cite{LOT18} gives an algorithm to obtain $\widehat{CFD}(S^3\backslash nb(K))$ from $CFK_{R}(K)$. (We do not specify a parametrization for the boudary of $S^3\backslash nb(K)$ as we will not need it.) Conversely, the construction in Section 4.3 of \cite{hanselman2018heegaard} can be easily modified to give an $\mathcal{R}$-complex from the bordered invariant $\widehat{CFD}(S^3\backslash nb(K))$. Note this $\mathcal{R}$-complex is locally equivalent to $CFK_{R}(K)$; see Proposition 2 of \cite{hanselman2019cabling}. Using this correspondence, we write $\widehat{CFD}(S^3\backslash nb(K))=\widehat{CFD}(C_{n-1})\oplus\widehat{CFD}(A_{n-1}).$ Then we have $$\widehat{CFD}(M(K))=\widehat{CFDA}(S^1 \times D^2\backslash nb(M))\boxtimes (\widehat{CFD}(C_{n-1})\oplus\widehat{CFD}(A_{n-1})).$$
Let $C_\mathcal{R}$ denote the reduced $\mathcal{R}$-complex obtained by edge-reduction of $\widehat{CFDA}(S^1 \times D^2\backslash nb(M))\boxtimes \widehat{CFD}(C_{n-1})$, and let $A$ denote the $\mathcal{R}$-complex corresponding to $\widehat{CFDA}(S^1 \times D^2\backslash nb(M))\boxtimes \widehat{CFD}(A_{n-1})$. Then $CFK_\mathcal{R}(M(K))=C_\mathcal{R}\oplus A$. As $C_{n-1}$ is a knot-like complex, we have $H_*(\widehat{C_\mathcal{R}})\cong \mathbb{F}$. This implies $H_*(\widehat{A})=0$ and hence $CFK_\mathcal{R}(M(K))$ is locally equivalent to $C_\mathcal{R}$. 

\begin{figure}[htb!]
\begin{center}
\includegraphics[scale=0.45]{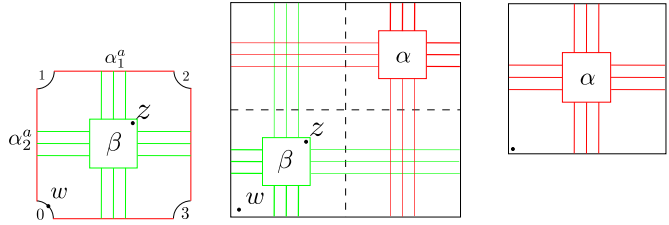}
\caption{An immersed-curve approach to compute $\widehat{C_{\mathcal{R}}}$. Left: a genus-one doubly-pointed bordered Heegaard diagram. Right: immersed curves for a type D structure. Middle: an immersed Heegaard diagram obtained from laying the bordered Heegaard diagram over the immersed-curve diagram.}\label{Figure, immersed curve pairing}
\end{center}
\end{figure}

$\widehat{C_\mathcal{R}}$ can be computed in terms immersed curves by an approach given in \cite{chen2019knot}: First represent $\widehat{CFD}(C_{n-1})$ as an immersed curve on the punctured torus using the algorithm given in \cite{hanselman2016bordered}, and denote this curve by $\alpha$. Then let $(\Sigma,\beta,\alpha^a_1,\alpha^a_2,w,z)$ be a genus-one doubly pointed bordered Heegaard diagram for the Mazur pattern. Laying the bordered Heegaard diagram over the immersed-curve diagram as shown in Figure \ref{Figure, immersed curve pairing}, we obtain a doubly-pointed Heegaard diagram $(T^2,\alpha,\beta,w,z)$ with $\alpha$ being an immersed Lagrangian. Then $\widehat{CFA}(S^1\times D^2,M)\boxtimes \widehat{CFD}(C_{n-1})$ is isomorphic to the filtered Lagrangian intersection Floer chain complex of $(T^2,\alpha,\beta,w,z)$. In particular, the reduced complex $\widehat{C_\mathcal{R}}$ corresponds to the Lagrangian intersection Floer chain complex of a minimal intersection diagram. For convenience, we work on the universal cover of $T^2$. As an example, a lift of the immersed-curve diagram corresponding to $\widehat{CFD}(C_{2})$ and a bordered Heegaard diagram for the Mazur pattern are shown in Figure \ref{Figure, bordered diagram for the Mazur pattern and immersed cuver for C(1,-2,2,-1)}.

\begin{figure}[htb!]
\begin{center}
\includegraphics[scale=0.45]{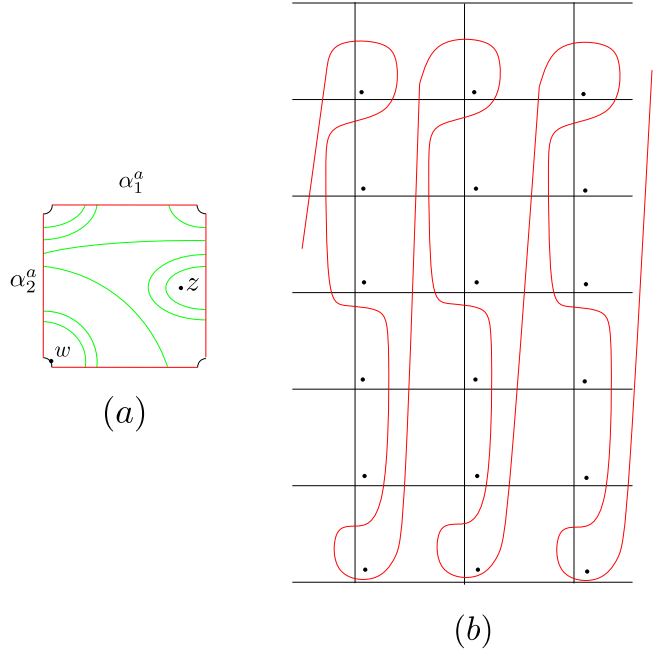}
\caption{(a) A bordered Heegaard diagram for the Mazur pattern. (b) The immersed-curve diagram corresponding to $C(1,-2,2,-1)$.}\label{Figure, bordered diagram for the Mazur pattern and immersed cuver for C(1,-2,2,-1)}
\end{center}
\end{figure}

\begin{figure}[htb!]
\begin{center}
\includegraphics[scale=0.73]{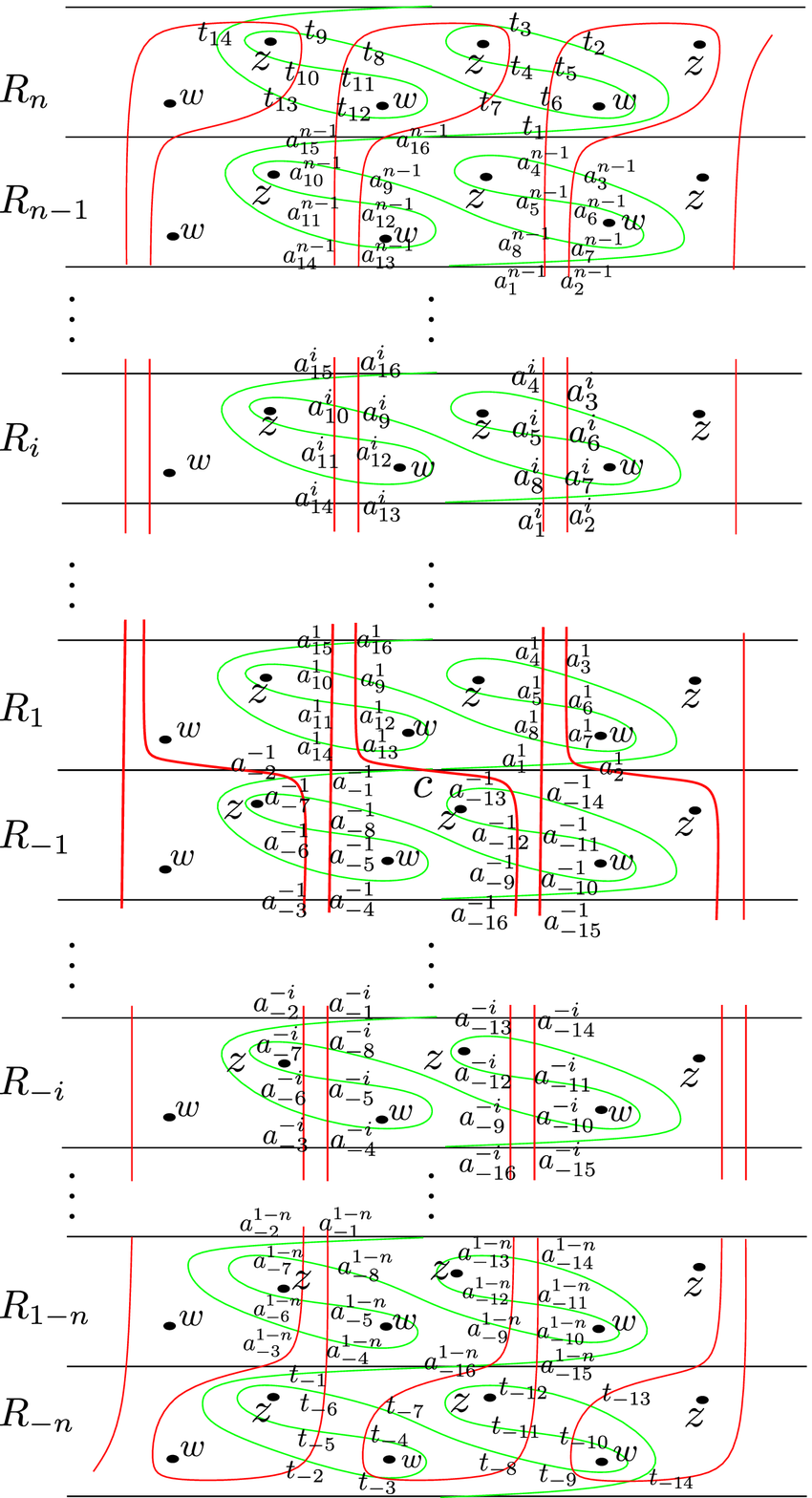}
\caption{The diagram that computes $\widehat{C_{\mathcal{R}}}$: The generators correspond to the intersection points, and differentials count disks that do not cross the $w$ base points.}\label{Figure, minimal intersection diagram}
\end{center}
\end{figure}
A minimal intersection diagram for $\widehat{C_{\mathcal{R}}}$ is shown in Figure \ref{Figure, minimal intersection diagram}. There are $2n$ rows. Observe that the diagram is symmetric about the intersection point $c$ in the center. Exploiting this symmetry, the upper $n$ rows are labeled from $R_1,\ldots, R_n$ and the lower $n$ rows are labeled from $R_{-1}$ to $R_{-n}$. Further notice that if we ignore $c$ in $R_1$, then the diagram for each of the rows from $R_1$ to $R_{n-1}$ are the same. There are $16$ intersection points in each row, and we label those in $R_i$ by $a^i_1\ldots a^i_{16}$ for $i=1,\ldots,n-1$, where the subscript increases as we traverse upwards along the $\beta$ curve. The top row $R_n$ has $14$ intersection points, and we label them by $t_1,\ldots,t_{14}$, where the subscript increases as we traverse upwards along the $\beta$ curve. Symmetrically, the intersection points in row $R_{-i}$ are labeled $a^{-i}_{-1},\ldots,a^{-i}_{-16}$ for $i=1,\ldots,n-1$ and the intersection points in $R_{-n}$ are labeled $t_{-1},\ldots,t_{-14}. $

\begin{figure}[htb!]
\begin{center}
\includegraphics[scale=0.7]{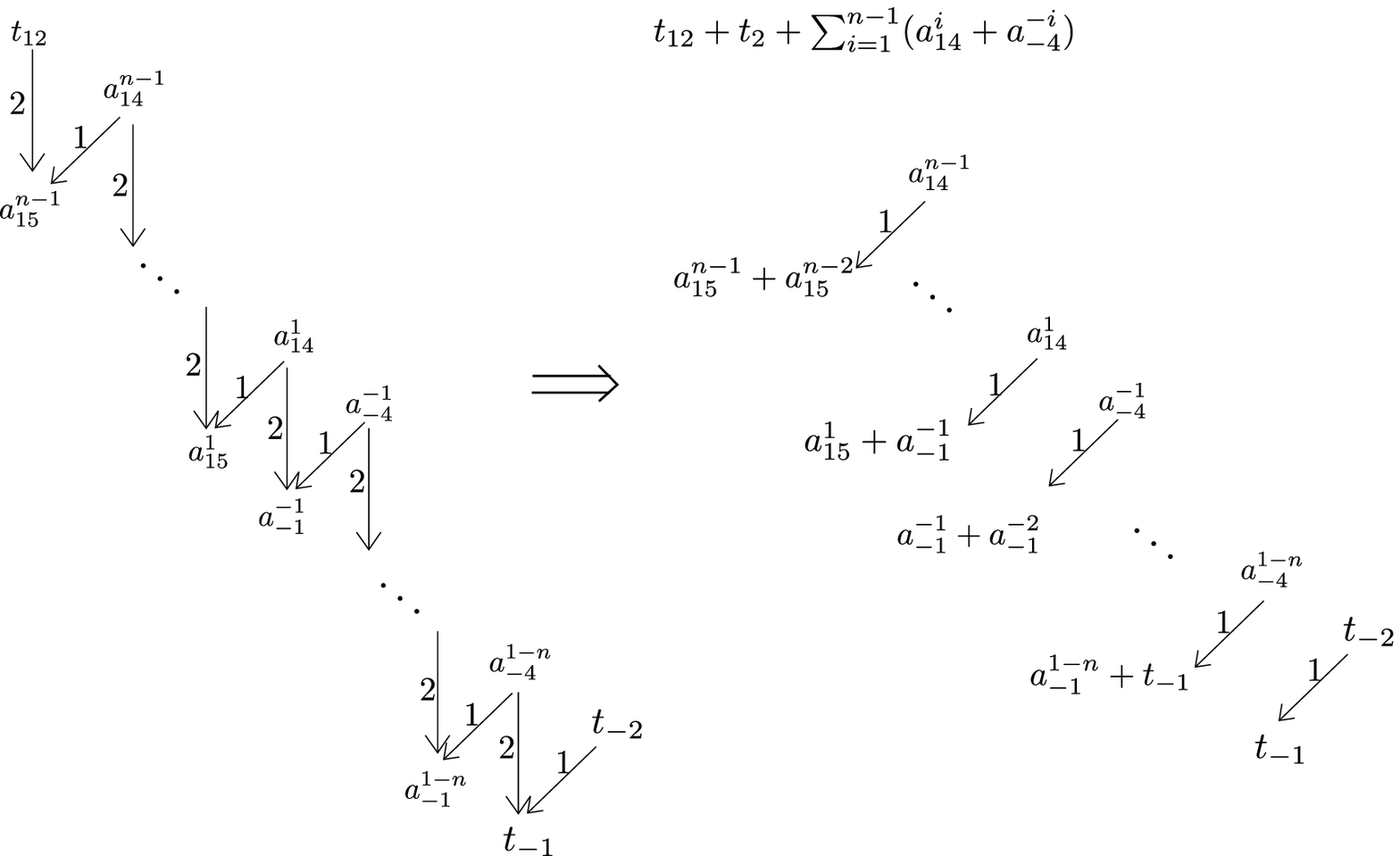}
\caption{The subcomplex containing $t_{12}$. The labels on the arrows indicate the lengths.}\label{Figure, the subcomplex containing t_12}
\end{center}
\end{figure}

\begin{table}[htb!]
\centering
{\renewcommand{\arraystretch}{1.2}
\begin{tabular}{|c|cc|c|c|cc|}
\cline{1-3} \cline{5-7}
           & A     & M      &  &          & A     & M       \\ \cline{1-3} \cline{5-7} 
$a^i_1$    & $i-1$ & $2i-3$ &  & c        & $0$   & $-2$    \\
$a^i_2$    & $i-1$ & $2i-4$ &  & $t_1$    & $n-1$ & $2n-3$  \\
$a^i_3$    & $i-2$ & $2i-5$ &  & $t_2$    & $n-2$ & $2n-4$  \\
$a^i_4$    & $i-2$ & $2i-4$ &  & $t_3$    & $n-2$ & $-3$    \\
$a^i_5$    & $i-1$ & $2i-3$ &  & $t_4$    & $n-1$ & $-2$    \\
$a^i_6$    & $i-1$ & $2i-4$ &  & $t_5$    & $n-1$ & $2n-3$  \\
$a^i_7$    & $i$   & $2i-3$ &  & $t_6$    & $n$   & $2n-2$  \\
$a^i_8$    & $i$   & $2i-2$ &  & $t_7$    & $n$   & $-1$    \\
$a^i_9$    & $i-1$ & $-3$   &  & $t_8$    & $n-1$ & $-2$    \\
$a^i_{10}$ & $i-1$ & $-2$   &  & $t_9$    & $n-1$ & $-2n-1$ \\
$a^i_{11}$ & $i$   & $-1$   &  & $t_{10}$ & $n$   & $-2n$   \\
$a^i_{12}$ & $i$   & $-2$   &  & $t_{11}$ & $n$   & $-1$    \\
$a^i_{13}$ & $i+1$ & $-1$   &  & $t_{12}$ & $n+1$ & $0$     \\
$a^i_{14}$ & $i+1$ & $0$    &  & $t_{13}$ & $n+1$ & $1-2n$  \\
$a^i_{15}$ & $i$   & $-1$   &  & $t_{14}$ & $n$   & $-2n$   \\
$a^i_{16}$ & $i$   & $-2$   &  &          &       &         \\ \cline{1-3} \cline{5-7} 
\end{tabular}
\vspace{5mm}
\caption{The Alexander and the Maslov grading of the intersection points in $R_1$ to $R_n$. Here $i=1,\ldots, n-1$. The gradings for the intersections in $R_{-1},\ldots, R_{-n}$ can be deduced from this table using symmetry:  If $x$ and $y$ are two intersection points symmetric about $c$, then $A(x)=-A(y)$ and $M(x)=M(y)-2A(y)$. }
\label{Table, Alexander and Maslov gradings of intersection points}
}
\end{table}
We can determine the Alexander grading and the Maslov grading of each intersection point. The relative grading differences are computed as in ordinary Lagrangian Floer chain complexes. (To slightly ease the tedious task of determining the relative gradings in our case, note if we fix a number $j\in\{1,\ldots,16\}$, then the relative grading difference between $a^i_j$ and $a^{i+1}_j$ is constant as $i$ varies.) We move to determine the absolute gradings. By symmetry of the knot Floer homology group, one can see the Alexander grading of $c$ is $0$ (see Page 31 of \cite{chen2019knot}). This fact together with the relative Alexander grading determine the absolute Alexander grading. We claim the Maslov grading of $t_{12}$ is $0$. In fact, the minimal subcomplex of $\widehat{C_\mathcal{R}}$ containing $t_{12}$ is shown in Figure \ref{Figure, the subcomplex containing t_12}, whose homology group is isomorphic to $\mathbb{F}$ and generated by the cycle $[t_{12}]+[a^{n-1}_{14}]+\cdots +[a^{1}_{14}]+[a^{-1}_{-4}]+\cdots +[a^{-n+1}_{-4}]+[t_{-2}]$. As $H_*(\widehat{C_\mathcal{R}})$ is isomorphic to $\mathbb{F}$ and is supported in Maslov grading $0$, we know $[t_{12}]+[a^{n-1}_{14}]+\cdots +[a^{1}_{14}]+[a^{-1}_{-4}]+\cdots +[a^{-n+1}_{-4}]+[t_{-2}]$ generates $H_*(\widehat{C_\mathcal{R}})$ and hence $M(t_{12})=0$. $M(t_{12})$ together with the relative Maslov grading determine the absolute Maslov grading. We list the gradings of the intersection points in $R_1$ to $R_n$ in Table \ref{Table, Alexander and Maslov gradings of intersection points}. The gradings for intersection points in $R_{-1}$ to $R_{-n}$ can be deduced from Table \ref{Table, Alexander and Maslov gradings of intersection points} by symmetries: $A(a^{-i}_{-j})=-A(a^i_j)$ and $M(a^{-i}_{-j})=M(a^i_j)-2A(a^i_j)$ for $i=1,\ldots,n-1$ and $j=1,\ldots,16$. $A(t_{-j})=-A(t_j)$ and $M(t_{-j})=M(t_{j})-2A(t_{j})$ for $j=1,\ldots,14$. 

Lemma \ref{Lemma, length of vertical arrows over simplied basis} (1) claims that over a vertically simplified basis the vertical arrows with terminals of Alexander grading $-n-1$ and Maslov index $-2n-2$ are of length $1$. To see this, note there is only one intersection point with Alexander grading $-n-1$ and Maslov index $-2n-2$; this is $t_{-12}$. The minimal subcomplex of $\widehat{C_{\mathcal{R}}}$ containing arrows to or from $t_{-12}$ is $t_{-11}\rightarrow t_{-12}$ and the length of this arrow is $1$. 

\begin{figure}[htb!]
\begin{center}
\includegraphics[scale=0.55]{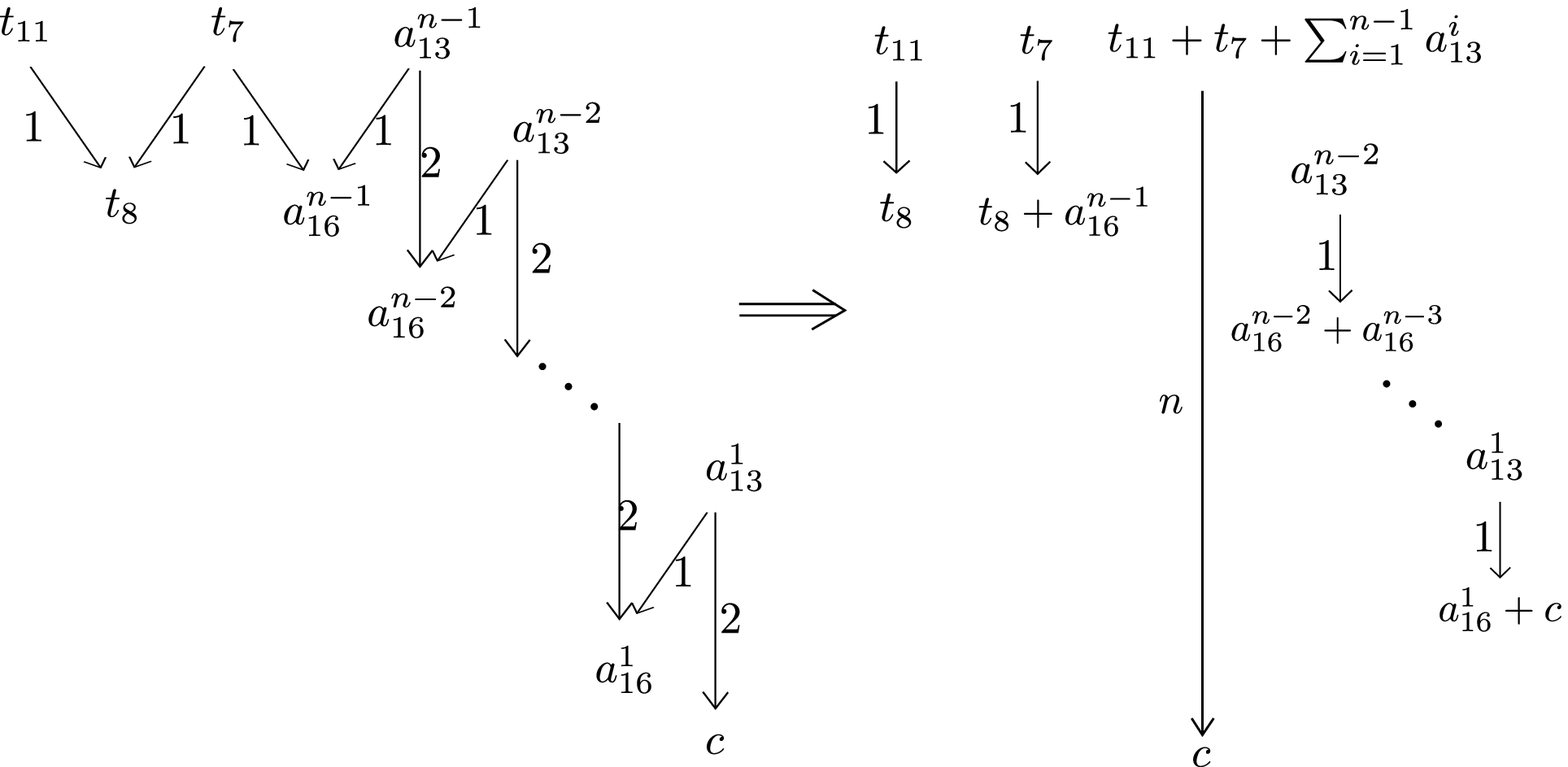}
\caption{}\label{Figure, subcomplex for (2)}
\end{center}
\end{figure}

Lemma \ref{Lemma, length of vertical arrows over simplied basis} (2) claims over a vertically simplified basis the vertical arrows with initials of Alexander grading $n$ and Maslov grading $-1$ are either of length $1$ or $n$, and there is only one such arrow of length $n$. The intersection points with Alexander grading $n$ and Maslov grading $-1$ are $t_{11}$, $t_7$, and $a^{n-1}_{13}$. The minimal subcomplex(es) of $\widehat{C_\mathcal{R}}$ which contain differentials initiating from these intersection points is (are) as shown in Figure \ref{Figure, subcomplex for (2)} (left). The claim is obvious after a filtered change of basis Figure \ref{Figure, subcomplex for (2)} (right).

\begin{figure}[htb!]
\begin{center}
\includegraphics[scale=0.55]{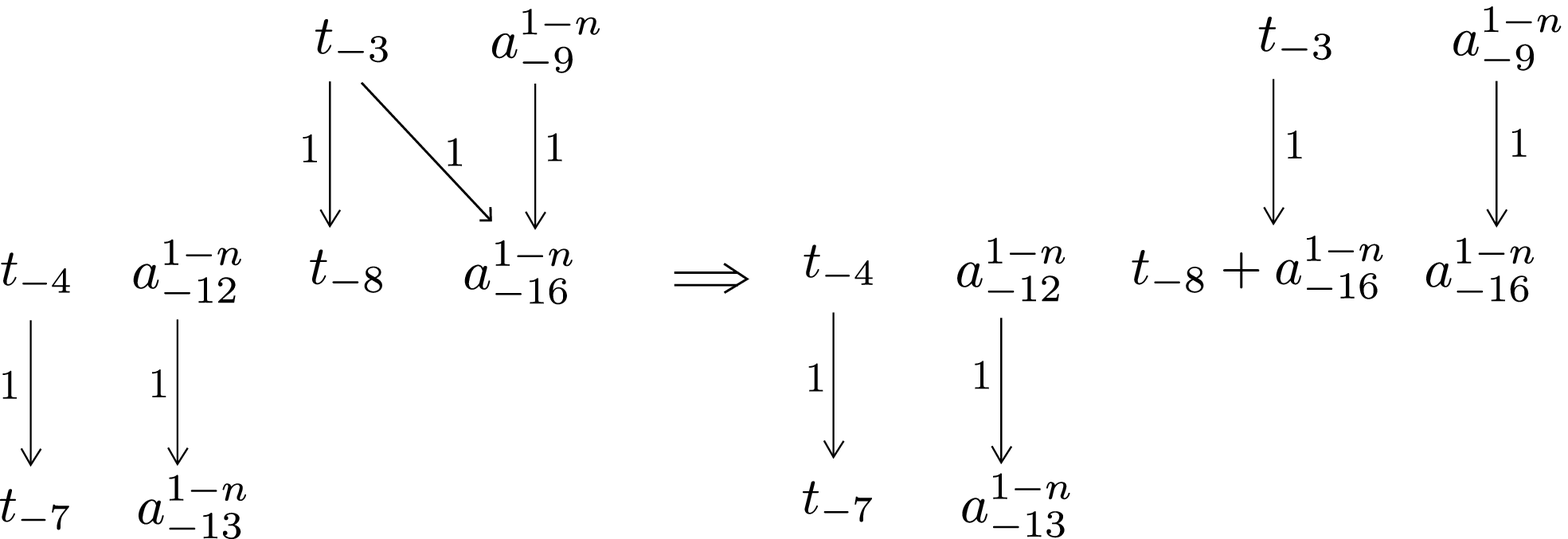}
\caption{}\label{Figure, subcomplex for (3)}
\end{center}
\end{figure}

Lemma \ref{Lemma, length of vertical arrows over simplied basis} (3) claims that the vertical arrows with initials or terminals of Alexander grading $-n+1$ and Maslov grading $-2n$ are of length $1$. Note the intersection points with Alexander grading $-n+1$ and Maslov grading $-2n$ are $t_{-4}$, $t_{-8}$, $a^{-n+1}_{-12}$, and $a^{-n+1}_{-16}$. The minimal subcomplexes of $\widehat{C_\mathcal{R}}$ containing these intersection points are shown in Figure \ref{Figure, subcomplex for (3)} (left). The claim is obvious after a filtered change of basis \ref{Figure, subcomplex for (3)} (right).

Lemma \ref{Lemma, length of vertical arrows over simplied basis} (4) claims there are no vertical arrows with initials of Alexander grading $n-2$ and Maslov grading $-3$, and the vertical arrows with terminals of Alexander grading $n-2$ and Maslov grading $-3$ are of length $1$.  Note the intersection points with Alexander grading $n-2$ and Maslov grading $-3$ are $t_{3}$ and $a^{n-1}_{9}$. The minimal subcomplexes containing these intersection points are $t_{4}\rightarrow t_{3}$ and $a^{n-1}_{12}\rightarrow a^{n-1}_9$, where both of the arrows are of length $1$. The claim follows.  
\begin{figure}[htb!]
\begin{center}
\includegraphics[scale=0.7]{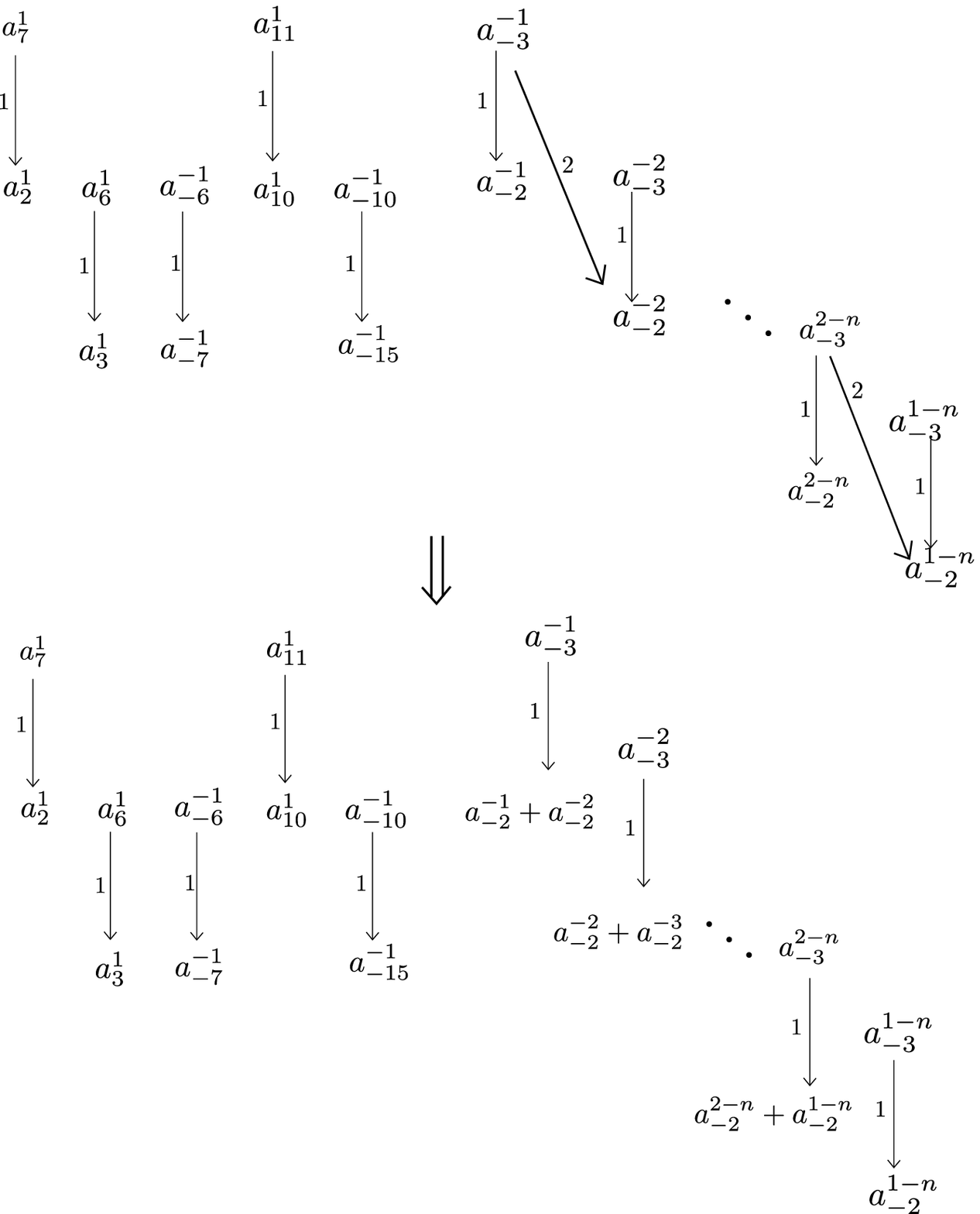}
\caption{}\label{Figure, subcomplex for (5)}
\end{center}
\end{figure}

Lemma \ref{Lemma, length of vertical arrows over simplied basis} (5) claims that over a vertically simplified basis the vertical arrows with terminals of Alexander grading $0$ and Maslov grading $-2$ are either of length $1$ or $n$, and the vertical arrows with initials of Alexander grading $0$ and Maslov grading $-2$ are of length $1$. To see this, note the intersection points with Alexander grading $0$ and Maslov grading $-2$ are $a^1_2$, $a^1_6$, $a^1_{10}$, $c$, $a^{-1}_{-2}$, $a^{-1}_{-6}$, and $a^{-1}_{-10}$. The minimal subcomplexes of $\widehat{C_\mathcal{R}}$ which contain differentials ending at $c$ appears in Figure \ref{Figure, subcomplex for (2)}. The minimal subcomplexes containing differential ending at the other intersection points are as shown in Figure \ref{Figure, subcomplex for (5)}. Claim (5) can be seen after a filtered change of basis.

Lemma \ref{Lemma, length of vertical arrows over simplied basis} (6) claims over a vertically simplified basis the vertical arrows with terminals of Alexander grading $1$ and Maslov grading $-1$ are of length $1$. The intersection points with Alexander grading $1$ and Maslov grading $-1$ are $a^{1}_{7}$, $a^{1}_{11}$, and $a^{1}_{15}$. The minimal subcomplexes of $\widehat{C_\mathcal{R}}$ containing $a^{1}_{7}$ and $a^{1}_{11}$ are shown in Figure \ref{Figure, subcomplex for (5)} and they do not admit incoming arrows. The subcomplex containing $a^{1}_{15}$ is shown in \ref{Figure, the subcomplex containing t_12} and the claim can be seen after a filtered change of basis. 

Lemma \ref{Lemma, length of vertical arrows over simplied basis} (7) claims over a vertically simplified basis there are no vertical arrows with initials of Alexander grading $-2$ and Maslov grading $-4$, and the vertical arrows with terminals of Alexander grading $-2$ and Maslov grading $-4$ are of length $1$. Note there is only one intersection point of Alexander grading $-2$ and Maslov grading $-4$; this is $a^{-1}_{-14}$. The claim follows readily from that the minimal subcomplex of $\widehat{C_\mathcal{R}}$ containing arrows to or from $a^{-1}_{-14}$ is $a^{-1}_{-11}\rightarrow a^{-1}_{-14}$, where the arrow has length $1$. 

Lemma \ref{Lemma, length of vertical arrows over simplied basis} (8) claims that over a vertically simplified basis the vertical arrows with initials or terminals of Alexander grading $-1$ and Maslov grading $-3$ are of length $1$. To see this, note the intersection points with Alexander grading $-1$ and Maslov grading $-3$ are $a^{1}_{3}$, $a^{-1}_{-7}$, $a^{-1}_{-11}$, and $a^{-1}_{-15}$. The minimal subcomplexes of $\widehat{C_\mathcal{R}}$ which contain $a^{1}_{3}$, $a^{-1}_{-7}$, and $a^{-1}_{-15}$ are already shown in Figure \ref{Figure, subcomplex for (5)}. The minimal subcomplex involving $a^{-1}_{-11}$ is observed in the previous paragraph. Claim (8) can be read off from these subcomplexes. 

Lemma \ref{Lemma, length of vertical arrows over simplied basis} (9) claims that over a vertically simplified basis there are no vertical arrows with terminals of Alexander grading $2$ and Maslov grading $0$, and the vertical arrows with initials of Alexander grading $2$ and Maslov grading $0$ are of length $1$. To see this, note the only intersection point of Alexander grading $2$ and Maslov grading $0$ is $a^{1}_{14}$. The minimal subcomplex of $\widehat{C_\mathcal{R}}$ which contain differentials initiating from $a^{1}_{14}$ is shown in Figure \ref{Figure, the subcomplex containing t_12}. Claim (4) can be seen after a filtered change of basis.

\end{proof}
 
\bibliographystyle{abbrv}
\bibliography{satellite}
\end{document}